\documentclass[10pt,a4paper]{article}
\usepackage{a4wide}
\usepackage{latexsym,stmaryrd}
\usepackage{amsmath,amssymb,amsthm,bm}
\usepackage[utf8]{inputenc}
\usepackage{t1enc, lmodern}
\usepackage{graphicx}
\usepackage{color}
\usepackage{algorithm, algorithmicx, algpseudocode}
\usepackage{fancyhdr}
\usepackage{pgf}
\usepackage{comment}

\newcommand{\rset}{\mathbb{R}}

\newcommand{\nset}{\mathbb{N}}
\newcommand{\nl}{\nolimits}

\newcommand{\ind}{\mathbf{1}}
\newcommand{\fl}{\longrightarrow}
\newcommand{\e}{\mathbb{E}}
\newcommand{\p}{\mathbb{P}}

\newcommand{\var}{\mathbb{V}}
\newcommand{\lp}{\mathrm{L}}

\newcommand{\m}{\mathcal}

\newcommand{\cov}{\mathbb{C}\mathrm{ov}}

\theoremstyle{plain}
\newtheorem{thm}{Theorem}[section]
\newtheorem{lemme}[thm]{Lemma}
\newtheorem{definition}[thm]{Definition}

\newtheorem{prop}[thm]{Proposition}

\newtheorem{hypo}[thm]{Hypothesis}

\theoremstyle{definition}

\theoremstyle{remark}
\newtheorem{rem}[thm]{Remark}
\makeatletter
\def\footnotestar#1{
  \insert \footins{\normalfont\footnotesize
    \par \noindent
    #1 \vskip0.2em
  }       
}         
\makeatother
\begin{document}
\title{\bf An asymptotical method to estimate the parameters of a
  deteriorating system under condition-based maintenance}
\author{Philippe Briand$^{1}$, Edwige Idée$^{2}$ and Céline Labart$^{3}$\vspace*{0.2cm}\\
  \vspace*{0.2cm}
{\small Laboratoire de Mathématiques, CNRS UMR 5127,
  Université de Savoie,}\\
  {\small Campus Scientifique,
73376 Le Bourget du Lac, France.}}
\footnotestar{$^1$ 
  \texttt{philippe.briand@univ-savoie.fr}.}
\footnotestar{$^2$
\texttt{edwige.idee@univ-savoie.fr}.}
\footnotestar{$^3$
\texttt{celine.labart@univ-savoie.fr}.}
\date{\today}

\maketitle

%
%
\begin{abstract} In this paper, we develop a new method to estimate the
  parameters of a deteriorating system under perfect condition-based maintenance. This
  method is based on the asymptotical behavior of the system, which is studied
  by using the renewal process theory. We obtain a Central Limit Theorem (CLT
  in the following) for the parameters. We compare the accuracy and the speed
  of the method with the maximum likelihood one (ML method in the following)
  on different examples.
\end{abstract}

\section{Introduction}\label{sec:introduction} Many systems suffer from
increasing wear with usage and age and are subject to failures resulting from
deteriorations. The deterioration and failures might lead to high costs, then preventive maintenance is necessary. In the past several decades,
maintenance, replacement and inspection problems have been widely studied (see
the surveys \cite{mac_call_65}, \cite{barlow_proschan_96}, \cite{pierskalla_voelker_76},
\cite{osaki_nagakawa_76} and \cite{sherif_smith_81},
\cite{valdes_flores_feldman_89} among others). If the deterioration of a
system can be observed while inspecting, it is more judicious to set up the
maintenance policy on the state of the system rather than on its
age. Deterioration systems and their optimal maintenance policy have
been studied in the literature (see \cite{mine_kawai_75},
\cite{ohnishi_mine_kawai_86}, \cite{tijms_85}, \cite{barbera_96},
\cite{wang_00} and \cite{grall_02}). 
In this paper, we consider a
system subject to random deteriorations, which can lead to failures. As long
as the system operates, it is monitored by planned inspections. At these
inspections, the system can be in two states : sane or damaged. If the system
is found out to be damaged, a preventive perfect repairing (``as good as
new'') is performed. If the system fails, an unplanned inspection is performed
immediately and a corrective perfect repairing (``as good as new'') is
done.\\

When the deterioration of the system can be continuously measured, the
deterioration is usually represented by a stochastic process with stationary
and independent increments and the states of the system are fixed by some
thresholds on the stochastic process (see \cite{grall_02}, \cite{grall_02_bis}
for example). However, some deteriorations can not be easily measured and the
state of the system is only known when inspecting. In order to deal with this
kind of problem, we assume that the transition time from repairing to
damaging and the transition time from damaging to failure are
two positive random variables, whose parameters $(\mu,\lambda)$ are unknowns (we consider
that failures only  ensue from deteriorations). Moreover, we
assume that there exists some uncertainty on planned inspection dates (which
can be due to physical or financial constraints). The purpose of the present
paper is to propose a new method, based on the asymptotic behavior of the
system, to estimate the unknown parameters of the transition times. At some time $t$, we assume that we only
know the number of inspections before $t$, and the state of the system at the
inspection dates. Then, we have at hand $N^r_t$, the
number of repairs before time $t$, $N^i_t$, the number of inspections before
time $t$, and $N^f_t$, the number of failures before time $t$. By using the
renewal processes theory, we write $(\mu,\lambda)$ as the limit of a function
of $(\frac{N^r_t}{t},\frac{N^i_t}{t},\frac{N^f_t}{t})$. Moreover, we get a CLT
for this triplet of variables, which gives us a confidence interval for
$(\mu,\lambda)$.\\

Practically, this method not only applies to long term systems. We can also
deal with identical units, repaired at least one time, operating independently
and simultaneously in a similar environment and being analogously
exploited. Putting repairing times of these units end to end is equivalent to
studying a single system on a long period of time. The paper is organized as
follows : Section \ref{sect:model_assumptions} introduces some notations and presents
the assumptions on the model, Section \ref{sect:renewal_theory} recalls
standard results on renewal and renewal reward processes. Section
\ref{sect:main_results} states a CLT for our parameters. Sections
\ref{sect:applications} and \ref{sect:num_ex} present some
applications and numerical examples.

\section{Model Assumptions}\label{sect:model_assumptions}

 In the following, we represent
the time from repairing to damaging by a positive random variable (r.v. in
the following) $Y^s$ and the time from damaging to failure by a positive
r.v. $Y^d$. We modelize the uncertainty of inspection dates (which can be due
to physical or financial constraints) by a sequence of random variables. 
  The elapsed time between two
inspection dates is a random variable $C$, and if $C_i$ denotes the time spent
between the $(i-1)$th and the $i$th inspection, $D_i:=C_1+\cdots+C_i$ is
the age of a system at the $i$th inspection (with convention
$D_0=0$). This enables to define the index of the inspection following the
damage
\begin{equation*}
	K^r = \inf\left\{ n \geq 1 : D_n \geq Y^s\right\} = 1+\sum_{n\geq 1} \ind_{D_n < Y^s}.
\end{equation*}
  and the inter-repairing time
\begin{align*}
  X^r=\min\left(D_{K_r}, Y^s+Y^d\right).
\end{align*}

This means that we repair the system as soon as a damage is detected
$(X^r=D_{K_r})$ (see Figure \ref{fig6}) or as soon as the system
fails $(X^r=Y^s+Y^d)$  (see Figure \ref{fig7}).  In the following, we assume that the law of $Y^s$
depends of a parameter $\mu$, the law of $Y^d$ depends of a parameter
$\lambda$, and the law of $C$ is known.\\

\begin{figure}[ht]
  \centering
  \includegraphics[width=15cm]{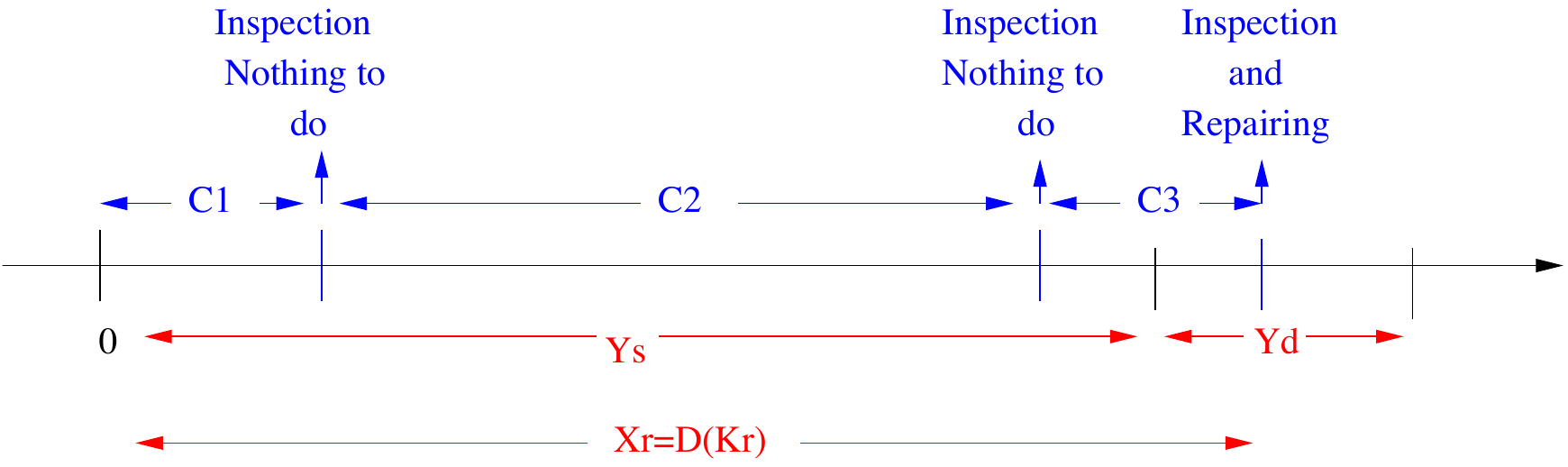}
  \caption{The system is repaired as soon as the damage is detected}
  \label{fig6}
\end{figure}

\begin{figure}[ht]
  \centering
  \includegraphics[width=15cm]{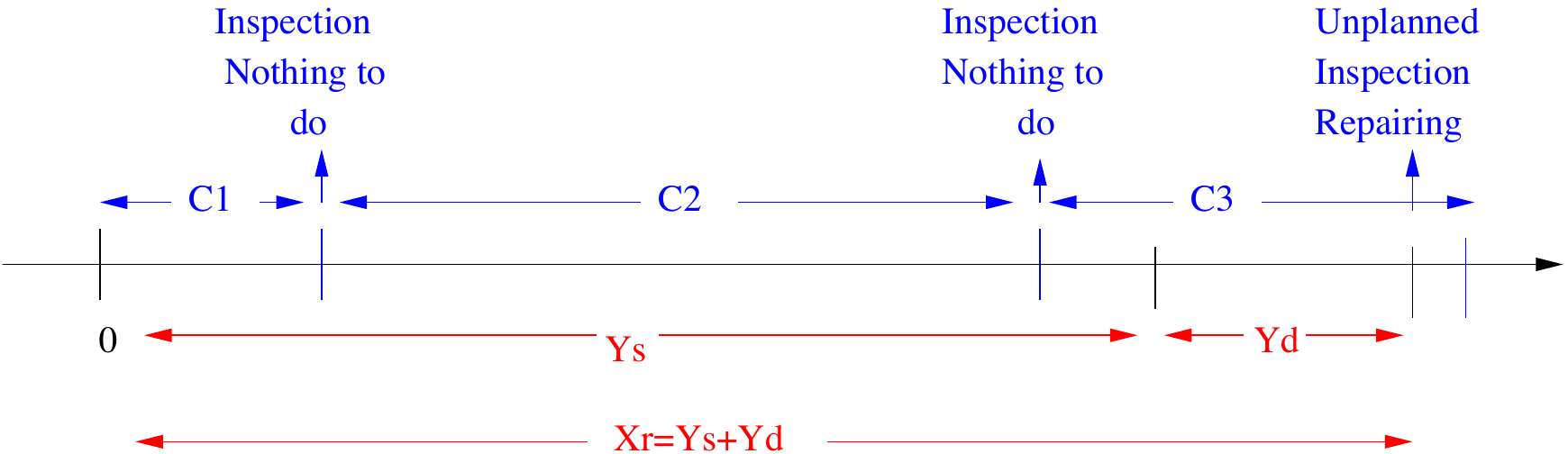}
  \caption{The system is repaired as soon as it fails}
  \label{fig7}
\end{figure}

We introduce the following notations :
\begin{definition}\hfill
\begin{itemize}
\item $F_s$ (resp. $F_d$) denotes the cumulative distribution function of $Y^s$ (resp. $Y^d$)
\item $R_s$ (resp. $R_d$) denotes the survival function of $Y^s$ (resp. $Y^d$)
\item $n_{\mu}$ denotes the law of $Y^s$
\item $L$ denotes the Laplace transform of $C$ : $\forall s \ge 0$, $L(s):=\e(e^{-sC})$
\item $\overline{B}(t)$ represents the index of the inspection following $t$ : $\overline{B}(t):=
  \inf\left\{ n \geq 1 : D_n \geq t\right\} = 1+ \sum_{n\geq 1} \ind_{D_n <
    t}$. Then, $K^r = \overline{B}\left(Y^s\right)$
\item $\underline{B}(t)$ represents the index of the inspection before $t$ :
  $\underline{B}(t):=
  \sup\left\{ n \geq 0 : D_n \le t\right\}$
\item $V^s$ denotes the age of the system when a damage is detected : $V^s=D_{K^r}$
\item $Z^d$ denotes the time of failure : $Z^d=Y^s+Y^d$
\item $P_d$ denotes the probability that the system fails before the
  inspection following the damage occurs : $P_d=\p(V^s\ge
  Z^d)=\p(D_{\overline{B}(Y^s)}-Y^s\ge Y^d)$
\item $m_x$ denotes $\e(X^r)$ and $m_k$ denotes $\e(K^r)$
\end{itemize}
With these
notations, $X^r=\min(V^s,Z^d)$.
\end{definition}

\begin{hypo}\label{hypo1} In the following, we assume that the r.v. $Y^s$ and $Y^d$ are
   such that the law of $Y^s$ depends on a
  parameter $\mu$ taking values in $U$ and the law of $Y^d$ depends on a
  parameter $\lambda$ taking values in $V$. 
  The r.v. $(C_i)_{i \ge 1}$ are independent with common law $C$, and such
  that $\p(C>0)=1$. We also assume that they are independent of $Y^s$ and $Y^d$. Moreover, we assume $\e((Y^s)^2)+\e(C^2)<\infty$.
\end{hypo}

\begin{rem} Under Hypothesis \ref{hypo1}, $\var(X^r)+\var(K^r)<\infty$. We
  first prove that $\e((K^r)^2)< \infty$.  $\e((K^r)^2)=\int_0^{\infty} \e(\overline{B}(t)^2)
  n_{\mu}(dt)=\int_0^{\infty} \frac{\e(\overline{B}(t)^2)}{t^2} t^2
  n_{\mu}(dt)$. $(\overline{B}(t))_{t\ge 0}$ is a renewal process with interarrival time
  $C$, then $\frac{\e(\overline{B}(t)^2)}{t^2}\rightarrow \frac{1}{(\e(C))^2}$. Since
   $\exists c>0$ such that $\forall t \ge 0$ $\e(\overline{B}(t)^2) \le c
   e^t$, we get $\forall \varepsilon >0$, $\exists \, t_0$ such that
   $\e((K^r)^2) \le ce^{t_0}+(\frac{1}{\e(C)^2}+\varepsilon)\e((Y^s)^2)$.
   Concerning $\e((X^r)^2)$, we have $\e((X^r)^2)\le
  \e((V^s)^2)$ and $V^s\le Y^s+C_{\overline{B}(Y^s)}$. Then,
  $ \e((V^s)^2)\le
  2(\e((Y^s)^2)+\e(C^2_{\overline{B}(Y^s)}))$. Since we have
  \begin{align*}
    \e(C^2_{\overline{B}(Y^s)})=\sum_{k=1}^{\infty} \e(C_k^2
    \ind_{\overline{B}(Y^s)=k})=\sum_{k=1}^{\infty} \e(C_k^2 \ind_{D_{k-1} <
      Y^s \le D_k})\le \sum_{k=1}^{\infty} \e(C^2)R^s(D_{k-1})=\e(C^2)\e(K^r),
  \end{align*}
  the result follows (the last equality comes from \eqref{eq:Krfond}).
\end{rem}

\begin{prop} Under Hypothesis \ref{hypo1}, we have
  \begin{align}
	\label{eq:Krfond}
	\e\left[K^r\right] & = \sum_{n\geq 0} \e\left[R_s(D_n)\right], \\
	\label{eq:Pdfond}
	1-P_d & =  \int_{0}^{+\infty} \e\left[R_d\left(D_{\overline{B}(t)}-t\right)\right]\,   n_{\mu}(dt) , \\
	\label{eq:Xrfond}
	\e\left[X^r\right] & = \e\left[Y^s\right] +  \int_{0}^{+\infty} \int_0^{+\infty} \p\left(D_{\overline{B}(t)}-t >y \right) R^{d}(y) \, dy\, n_{\mu}(dt).
  \end{align}
  If in addition we assume that $Y^d$ follows the exponenial law, we get
  \begin{align}
    \label{eq:Pdexp}
    1-P_d &= \left(1-L(\lambda)\right) \sum_{k\geq 1} \e\left[ e^{-\lambda
        D_k}\, \int_0^{D_k} e^{\lambda t} n_{\mu}(dt)
    \right] -L(\lambda) n_{\mu}\left(\{0\}\right),\\
    \label{eq:Xrexp}
    \e\left[X^r\right] &= \e\left[Y^s\right] + \frac{1}{\lambda} \p\left(V^s \geq Z^d\right) = \e\left[Y^s\right] + \frac{1}{\lambda} P_d.
  \end{align}
\end{prop}

\begin{proof} By independence, one gets \eqref{eq:Krfond} : 
  $$\e\left[K^r\right] = 1 + \sum_{n\geq
    1} \e\left[\ind_{D_n < Y^s}\right] = 1 + \sum_{n\geq 1}
  \e\left[R_s(D_n)\right] = \sum_{n\geq 0} \e\left[R_s(D_n)\right].
  $$
  Let us
  prove \eqref{eq:Pdfond}. By independence, we have
  \begin{equation*}
	P_d  = \p\left(D_{\overline{B}(Y^s)}-Y^s \geq Y^d\right) = \int_0^{+\infty} \p\left( D_{\overline{B}(t)}-t \geq Y^d\right) n_{\mu}(dt)=\int_{0}^{+\infty} \e\left[F_d\left(D_{\overline{B}(t)}-t\right)\right]  n_{\mu}(dt),
  \end{equation*}
  and the result follows. If $Y^d$ follows an exponential law of parameter
  $\lambda$,
\begin{align*}
	1-P_d & = \int_0^{+\infty} \e\left[e^{-\lambda \left(D_{\overline{B}(t)}-t\right)}\right] \, n_{\mu}(dt)  = \int_0^{+\infty} \sum_{k\geq 1} \e\left[ e^{-\lambda \left(D_k-t\right)} \ind_{\overline{B}(t)=k}\right] \, n_{\mu}(dt), \\
	& = \int_0^{+\infty} \sum_{k\geq 1} \e\left[ e^{-\lambda \left(D_k-t\right)} \ind_{D_{k-1}<t\leq D_k}\right] \, n_{\mu}(dt), \\
	& = \sum_{k\geq 1} \e\left[ \int_0^{+\infty} e^{-\lambda \left(D_k-t\right)} \left(\ind_{t\leq D_k} - \ind_{t\leq D_{k-1}}\right) \, n_{\mu}(dt)\right], \\
	& = \sum_{k\geq 1} \e\left[ e^{-\lambda D_k}\, \int_0^{+\infty} e^{\lambda t} \ind_{t\leq D_k} \, n_{\mu}(dt) - e^{-\lambda C_k} e^{-\lambda D_{k-1}} \int_0^{+\infty} e^{\lambda t} \ind_{t\leq D_{k-1}} \, n_{\mu}(dt)\right].
	\end{align*}
By independence,
	\begin{align*}
		1-P_d & = \sum_{k\geq 1} \left(\e\left[ e^{-\lambda D_k}\, \int_0^{+\infty} e^{\lambda t} \ind_{t\leq D_k} \, n_{\mu}(dt) \right] - L(\lambda)\,  \e\left[ e^{-\lambda D_{k-1}} \int_0^{+\infty} e^{\lambda t} \ind_{t\leq D_{k-1}} \, n_{\mu}(dt)\right]\right),
	\end{align*}
	and \eqref{eq:Pdexp} follows.

  In order to prove \eqref{eq:Xrfond}, we compute the
  survival function of $X^r$.
  \begin{align*}
	\p\left(X^r > x\right) & = \p\left(\min\left(V^s,Z^d\right) > x\right) = \p\left( D_{\overline{B}(Y^s)}> x, Y^s + Y^d > x\right), \\
	& = \p\left( Y^s > x\right) + \p\left( D_{\overline{B}(Y^s)}> x, Y^s + Y^d > x, Y^s \leq x\right), \\
	& = R_s(x) + \e\left[ \ind_{D_{\overline{B}(Y^s)}>x} \ind_{Y^s\leq x}
      R_d\left(x-Y^s\right)\right],\\
    & = R_s(x) + \int_0^{+\infty} \p\left(D_{\overline{B}(t)}-t >x-t \right) R^{d}(x-t)\ind_{x-t >0} \, n_{\mu}(dt).
  \end{align*}
  Then,
  \begin{align*}
	\e\left[X^r\right] &= \int_{0}^{+\infty} \p\left( X^r > x\right) dx \\
	& = \e\left[Y^s\right] +  \int_{0}^{+\infty} \int_0^{+\infty} \p\left(D_{\overline{B}(t)}-t >x-t \right) R^{d}(x-t)\ind_{x-t >0} \, n_{\mu}(dt) \, dx,\\
	&= \e\left[Y^s\right] +  \int_{0}^{+\infty} \int_0^{+\infty} \p\left(D_{\overline{B}(t)}-t >x-t \right) R^{d}(x-t)\ind_{x-t >0} \, dx \, n_{\mu}(dt) ,
\end{align*}
and the change of variable $y=x-t$ gives the result. If $Y^d$ follows the
exponential law with density $f_d$, $R_d(y)=\frac{f^d(y)}{\lambda}$, and \eqref{eq:Xrexp} follows.

\end{proof}

\section{Renewal Theory}\label{sect:renewal_theory}
\subsection{Renewal and Renewal Reward processes}
In this section, we recall classical results on renewal processes. We refer to
\cite{cocozza_97}, \cite{asmussen_03} \cite[Appendix B]{aven_jensen_99} among others. Let
$(X_n)_{n \ge 0}$ be a sequence of nonnegative independent identically
distributed (i.i.d.) r.v. with distribution $F$. We assume that $\p(X=0)<1$.
We also define the sequence $(T_n)_n$
\begin{align*}
  T_0=0,\;\; T_n=\sum_{i=1}^n X_i, \; n \in \nset^*
\end{align*}
and
\begin{align*}
  N_t=\sup\{j : t_j \le t\}=\sum_{i=1}^{\infty} \ind_{T_j \le t}
\end{align*}
 $(N_t)_{t\ge 0}$
is a renewal process, and we have
\begin{thm}[Almost sure convergence, $\lp^1$ convergence and CLT for $N_t$]\label{thm1}
 We have
  \begin{align*}
   \lim_{t\rightarrow \infty} \frac{N_t}{t} = \frac{1}{\e(X)} \mbox{a.s. and
     in $\lp^1$.}
 \end{align*}
  Assume that $ \var(X) < \infty$. Then, $N_t$ satisfies the following CLT
  \begin{align*}
     \lim_{t\rightarrow \infty} \sqrt{t}
    \left(\frac{N_t}{t}-\frac{1}{\e(X)}\right)\stackrel{law}{=}\mathcal{N}\left(0,\frac{\var(X)}{\e(X)^3}\right).
  \end{align*}
\end{thm}

Let $(X_n,Y_n)_{n \ge 1}$ denote a sequence of i.i.d. pairs of r.v.. $Y_i$ can
be interpreted as the reward associated with the $j^{th}$ interarrival time $X_j$
($Y_j$ may depend on $X_j$). The
process $Z_t$ defined by
\begin{align*}
  Z_t=\sum_{i=1}^{N_t} Y_i 
\end{align*}
satisfies the following theorem
\begin{thm}[Almost sure convergence, $\lp^1$ convergence and CLT for $Z_t$]\label{thm2}
   We have
  \begin{align*}
   \lim_{t\rightarrow \infty} \frac{Z_t}{t} = \frac{\e(Y)}{\e(X)} \mbox{
     almost surely and in $\lp^1$}.
  \end{align*}
  Assume that $ \var(Y) < \infty$ and $ \var(X) < \infty$. Then, $Z_t$ satisfies the following CLT
  \begin{align*}
     \lim_{t\rightarrow \infty} \sqrt{t}
    \left(\frac{Z_t}{t}-\frac{\e(Y)}{\e(X)}\right)\stackrel{law}{=}\mathcal{N}\left(0,\frac{\var(Y-\frac{\e(Y)}{\e(X)}X)}{\e(X)^2}\right).
  \end{align*}
\end{thm}

The second part of Theorem \ref{thm2} ensues from Rényi's theorem, recalled
below, which will be useful in the following.

\begin{thm}[\cite{renyi_57}, Theorem 1]
  Let $(\xi_n)_{n\geq 1}$ be a sequence of i.i.d. square integrable
  r.v. such that $\sigma^2 = \e[\xi^2]$. Let $(\nu_t)_{t\ge 0}$ be a process
  taking values in $\nset$ such that $\nu_t/t$ converges in
  probability when $t\rightarrow \infty$ to a constant $c>0$. Then
  \begin{equation*}
    \dfrac{\sum_{i=1}^{\nu_t} \xi_i}{\sqrt{\nu_t}} \fl \m N(0,\sigma^2),\quad \dfrac{\sum_{i=1}^{\nu_t} \xi_i}{\sqrt{t}} \fl \m N(0,c\,\sigma^2).
  \end{equation*}
\end{thm}

\subsection{Age and Survival of a renewal process}
Let $(A_t,S_t)_{t \ge 0}$ denotes the age and the survival of a renewal
process : $A_t:=t-T_{N_t}$ and $S_t:=T_{N_t+1}-t$. The following result gives
the limiting laws of both processes :
\begin{prop}\label{prop1} Let us consider a renewal process with interarrival time $X$ such
  that $\e(X)<\infty$. Then
  \begin{align*}
  &\lim_{t\rightarrow \infty} \p(A_t \le x)=\lim_{t\rightarrow \infty} \p(S_t \le x)=\frac{1}{\e[X]}\int_0^x
  \p(X>u)du,\\
  &\mbox{and their densities are given by } p_{A_{\infty}}(x)=p_{S_{\infty}}(x)=\frac{1}{\e[X]}\p(X>x).
  \end{align*}
\end{prop}

\begin{lemme}\label{lem4}
    The process $(\frac{\underline{B}(A_t)}{\sqrt{t}})_{t \ge 0}$ converges in probability to $0$ when
    $t$ tends to infinity.
  \end{lemme}
  \begin{proof}
  Let us first prove that $(\frac{A_t}{\sqrt{t}})_{t \ge 0}$ converges in probability to $0$ when
    $t$ tends to infinity. Indeed, by using the Markov
 inequality, we get $\forall \varepsilon<1$,
 $\p\left(\frac{A_t}{\sqrt{t}}>\varepsilon\right)=\p\left(\frac{A_t}{\sqrt{t}}\wedge 1>\varepsilon\right) \le
 \frac{\e\left[\frac{A_t}{\sqrt{t}} \wedge 1\right]}{\varepsilon}$. Since the
 age of a renewal process converges in law when $t\to +\infty$ (see Proposition \ref{prop1}), we get
\begin{equation*}
	\forall a >0, \qquad \limsup_{t \rightarrow \infty} \e\left[\frac{A_t}{\sqrt{t}} \wedge 1\right] \leq \limsup_{t \rightarrow \infty} \e\left[\frac{A_t}{\sqrt{a}} \wedge 1\right] = \e\left[\frac{A_\infty}{\sqrt{a}} \wedge 1\right] \leq \frac{\e[A_\infty]}{\sqrt{a}}.
  \end{equation*}
  Then, $\forall \varepsilon<1$, $\forall a >0$, $\limsup_{t \rightarrow
    \infty} \p\left(\frac{A_t}{\sqrt{t}}>\varepsilon\right)\le
  \frac{\e[A_\infty]}{\varepsilon\sqrt{a}}$ and the result follows. Let us now
  prove the Lemma. $\forall \varepsilon <1$, 
  \begin{align*}
  \p\left(\frac{\underline{B}(A_t)}{\sqrt{t}}>\varepsilon\right)&=\p\left(\underline{B}(A_t)>\varepsilon
    \sqrt{t}\right)=\p\left(\sum_{i=1}^{\lceil \varepsilon \sqrt{t} \rceil}
    C_i \le A_t\right)=\p\left(\frac{\varepsilon}{\varepsilon \sqrt{t}}\sum_{i=1}^{\lceil \varepsilon \sqrt{t} \rceil}
    C_i \le \frac{A_t}{ \sqrt{t}}\right).
\end{align*}
Let $\eta >0$. It remains to split the last probability in two parts by introducing the set $\{ \left|\frac{1}{ \varepsilon \sqrt{t}}\sum_{i=1}^{\lceil \varepsilon \sqrt{t} \rceil}
    C_i -\e(C)\right| > \eta\}$ and its complement. The strong law of large
  numbers and the convergence in probability of $(\frac{A_t}{\sqrt{t}})_{t \ge 0}$ end the proof. 
\end{proof}

\section{Main Results}\label{sect:main_results}
\subsection{Number of repairs, number of failures and number of
  inspections} 

\begin{definition}
  Since the system is repaired as good as new at each repairing date, the
number of repairs at time $t$ is a renewal process given by
\begin{align*}
  N^r_t=\sum_{i=1}^{\infty} \ind_{T^r_i \le t},
\end{align*}
where $T^r_i=\sum_{j=1}^i X^r_j$.
  
\end{definition}

\begin{definition} Let $X^f$ be the r.v. representing the time between two
  consecutive failures :
\begin{equation*}
	X^f \stackrel{law}{=} \sum_{i=1}^\tau X_i,\quad\text{ where } \tau=\inf\left\{i\geq 1 :
      V^s_i \ge Z^d_i \right\}\quad \inf\emptyset=+\infty.
  \end{equation*}
  The number of failures at time $t$ is given by
  \begin{align*}
    N^f_t=\sum_{i=1}^{\infty} \ind_{T^f_i \le t},
  \end{align*}
  where $T^f_i=\sum_{j=1}^i X^f_j$. Then, $N^f_t$ is a renewal process
  with interarrival time $X^f$.
\end{definition}

\begin{rem}
  $N^f_t$ is also a renewal reward process, since $N^f_t=\sum_{i=1}^{N^r_t}
  \ind_{V^s_i\ge Z^d_i}$.
\end{rem}

\begin{rem}The number of damages $N^d$ is a renewal process with
  interarrival time $X^d \stackrel{law}{=} \sum_{i=1}^\sigma X_i,\quad\text{ where } \sigma=\inf\left\{i\geq 1 :
      Z^d_i > V^s_i \right\}$. As for $N^f$, we can write $N^d_t=\sum_{i=1}^{N^r_t}
  \ind_{Z^d_i\ge V^s_i}$.
\end{rem}

\begin{definition}
The number of inspections at time $t$, denoted $N^i_t$, is given by
\begin{equation*}
	N^i_t = \sum_{i=1}^{N^r_t} K^r_i + \underline{B}(t-T_{N^r_t}),
  \end{equation*}
  where $K^r_i$ denotes the number of inspections on the interval $]T^r_{i-1},T^r_i]$.
\end{definition}

\begin{thm}[Almost sure and $\lp^1$ convergences for $\frac{N^r_t}{t}$, $\frac{N^f_t}{t}$
  and $\frac{N^i_t}{t}$]\label{thm3}
  The following results hold almost surely and in $\lp^1$
  \begin{align}\label{eq:LGN}
    \lim_{t\to+\infty} \frac{N^r_t}{t} = \frac{1}{m_x}, \;
    \;\lim_{t\to+\infty} \frac{N^f_t}{t} = \frac{P_d}{m_x}, \mbox{ and }
	\lim_{t\to+\infty} \frac{N^i_t}{t} =
    \frac{m_k}{m_x}.
  \end{align}
\end{thm}

\begin{proof}
  The first and second results ensue from Theorem \ref{thm1} and from Wald's
  identity, since $\e(X^f)=\e(\tau)\e(X^r)$ and $\tau$ has a geometric law of
  parameter $P_d$. 
  From the definition of $N^i_t$, we get $\sum_{i=1}^{N^r_t} K_i \leq N^i_t \leq \sum_{i=1}^{N^r_t+1} K_i$. Then,
  \begin{align*}
     \frac{1}{N^r_t}\sum_{i=1}^{N^r_t} K^r_i \leq \frac{N^i_t}{N^r_t} \leq \frac{N^r_t+1}{N^r_t} \frac{1}{N^r_t+1}\sum_{i=1}^{N^r_t+1} K^r_i.
   \end{align*} Since $\lim_{t \rightarrow \infty} N^r_t=\infty$ almost surely,
   we get $\lim_{t \rightarrow \infty} \frac{N^i_t}{N^r_t} =\e(K^r)$. The strong
   law of large numbers for renewal processes yields the almost sure
   convergence of $\frac{N^i_t}{t}$.\\
   Let us now deal with the $\lp^1$-convergence. We have
   \begin{align}\label{eq3}
     \e[\sum_{i=1}^{N^r_t}
     K^r_i] \leq \e[N^i_t] \leq \e[\sum_{i=1}^{N^r_t+1} K^r_i].
   \end{align}
   We first deal with the right hand side. $N^r_t +1$ is an $\m
   F$-stopping time ($\{N^r_t+1 = k\} = \{N^r_t = k-1\} = \{T^r_{k-1}\leq t <
   T^r_k\}$), then
\begin{equation}\label{eq4}
	\e\left[N^i_t\right]\leq \e\left[\sum_{k=1}^{N^r_t+1} K^r_i\right] = \e\left[N^r_t + 1\right] \, \e\left[K^r_i\right].
\end{equation}
Concerning the left hand side, we write 
\begin{equation*}
	\e\left[\sum_{k=1}^{N^r_t} K^r_i\right] = \e\left[\sum_{k\geq 1} K^r_i \ind_{T_k \leq t}\right] = \sum_{k\geq 1}\e\left[ K^r_i\ind_{T^r_k \leq t}\right].
  \end{equation*}
  By symetry, we get $	\e\left[\sum_{k=1}^{N^r_t} K^r_i\right] = \e\left[ K^r_1
    \sum_{k\geq 1} \ind_{T_k \leq t}\right] =  \e\left[K^r_1
    N^r_t\right]$. Combining this result with \eqref{eq3} and \eqref{eq4}, we obtain
\begin{equation*}
	\e\left[ K^r_1 N^r_t\right] \leq \e\left[N^i_t\right] \leq \e\left[N^r_t +
      1\right] \, \e\left[K^r \right].
\end{equation*}
Since $\lim_{t\rightarrow \infty} N^r_t/t \fl 1/\e\left[X^r\right]$, Fatou's lemma
gives $\liminf_{t\to+\infty} \frac{1}{t}\, \e\left[ K^r_1 N^r_t\right] \geq
\frac{\e\left[K^r\right]}{\e\left[X^r\right]}$. The elementary renewal theorem gives $ \lim_{t\to+\infty} \frac{1}{t}
\e\left[N^r_t + 1\right] \, \e\left[K^r\right] =
\frac{\e\left[K^r\right]}{\e\left[X^r\right]}$ and the last result follows.
\end{proof}

\begin{thm}[CLT for $\frac{N^r_t}{t}$, $\frac{N^f_t}{t}$
  and $\frac{N^i_t}{t}$]\label{thm_TCL} Let us denote $X:=X^r$, $I:=\ind_{V^s
    \ge Z^d}$, $K:=
  K^r$.
  The following result holds
  \begin{equation*}
	\sqrt{t} \left(\frac{N^r_t}{t}-\frac{1}{m_x}, \frac{N^f_t}{t} -
      \frac{P_d }{m_x}, \frac{N^i_t}{t}-\frac{m_k}{m_x}\right) \fl \m
    N(0,R),
  \end{equation*}
  where
  \begin{align*}
    R =  (m_x)^{-3} \begin{pmatrix} \var[X] & \cov [ X, P_d X - m_x I] & \cov [ X, m_k X - m_x K]  \\
	\cov [ X, P_d X - m_x I] & \var[P_d X - m_x I] &
    \cov [ P_d X - m_x I, m_k X - m_x K]\\
    \cov [ X, m_k X - m_x K] &  \cov [ P_d X - m_x I,
    m_k X - m_x K] &  \var[m_k X - m_x K]\\
	\end{pmatrix} .
\end{align*}
\end{thm}

\begin{proof}
  Let $(X_n,Y_n,Z_n)$ be a sequence of i.i.d. square integrable r.v. in
  $\rset_+^3$. One denotes $\lambda_x=\e[X]$, $\lambda_y = \e[Y]$ and
  $\lambda_z = \e[Z]$. We denote $S^x_n=X_1+\ldots+X_n$,
  $S^y_n=Y_1+\ldots+Y_n$, $S^z_n=Z_1+\ldots+Z_n$ and $N^x_t = \sup\{ n\geq 1 :
  S^x_n \geq t\}$.\\

We consider the triplet $Q_n = \left(S^x_n -n \lambda_x, \lambda_x\, S^y_n
  - \lambda_y\, S^x_n, \lambda_x\, S^z_n
  - \lambda_z\, S^x_n \right)$. CLT gives us that $	\frac{1}{\sqrt{n}} Q_n \fl
\m N (0, V)$ where
\begin{equation*}
	V = \begin{pmatrix} \var[X] & \cov [ X, \lambda_x Y - \lambda_y X] & \cov [ X, \lambda_x Z - \lambda_z X] \\
	\cov [ X, \lambda_x Y - \lambda_y X] & \var[\lambda_x Y - \lambda_y X]
    &\cov [\lambda_x Y - \lambda_y X, \lambda_x Z - \lambda_z X]\\
	\cov [ X, \lambda_x Z - \lambda_z X] & \cov [\lambda_x Y - \lambda_y X,
    \lambda_x Z - \lambda_z X] & \var[\lambda_x Z - \lambda_z X]
	\end{pmatrix}	
\end{equation*}
By applying Rényi's theorem (see Theorem \ref{thm3}) to the real r.v.  $u\cdot
Q_{N^x_t}/\sqrt{N^x_t}$ we obtain
\begin{equation*}
	\frac{1}{\sqrt{N^x_t}}Q_{N^x_t} \fl \m N(0,V), \mbox{ and }\qquad \frac{1}{\sqrt{t}} Q_{N^x_t} \fl \m N(0, \lambda_x^{-1}\, V).
  \end{equation*}
  Let $A^x_t$ denote the age of the renewal process $N^x$,
  i.e. $A^x_t:=t-S^x_{N^x_t}$. We have 
\begin{align*}
	Q_{N^x_t} &= \left( S^x_{N^x_t} - \lambda_x N^x_t , \lambda_x\,
      S^y_{N^x_t} - \lambda_y\, S^x_{N^x_t},  \lambda_x\,
      S^z_{N^x_t} - \lambda_z\, S^x_{N^x_t}\right)\\
    &= \left(t-\lambda_x N^x_t, \lambda_x\, S^y_{N^x_t} - \lambda_y t,
      \lambda_x\, \left(S^z_{N^x_t}+\underline{B}(A^x_t) \right) - \lambda_z t\right) +
    (1,\lambda_y, \lambda_z) A^x_t-(0,0,\lambda_x) \underline{B}(A^x_t) .
\end{align*}
Combining Lemma \ref{lem4} and Slutsky's Lemma yields 
\begin{equation*}
	\frac{1}{\sqrt{t}} \left(t-\lambda_x N^x_t, \lambda_x\, S^y_{N^x_t} -
      \lambda_y t, \lambda_x\, \left(S^z_{N^x_t}+ \underline{B}(A^x_t) \right) -
      \lambda_z t\right) \fl \m N(0, \lambda_x^{-1} V)
\end{equation*}
Combining this result and the application $(x,y,z)\longmapsto
(-x,y,z)/\lambda_x$ gives
\begin{equation*}
	\frac{1}{\sqrt{t}} \left(N^x_t-\frac{t}{\lambda_x}, S^y_{N^x_t} -
      \frac{\lambda_y t}{\lambda_x},S^z_{N^x_t}+\underline{B}(A^x_t) -
      \frac{\lambda_z t}{\lambda_x} \right) \fl \m N(0, R),
  \end{equation*}
  where
  $$ R =  \lambda_x^{-3}\begin{pmatrix} \var[X] & \cov [ X, \lambda_y X - \lambda_x Y] & \cov [ X, \lambda_z X - \lambda_x Z] \\
	\cov [ X, \lambda_y X - \lambda_x Y] & \var[\lambda_y X - \lambda_x Y]
    &\cov [\lambda_y X - \lambda_x Y, \lambda_z X - \lambda_x Z]\\
	\cov [ X, \lambda_z X - \lambda_x Z] & \cov [\lambda_y X - \lambda_x Y,
    \lambda_z X - \lambda_x Z] & \var[\lambda_z X - \lambda_x Z]
  \end{pmatrix}.$$

  To conclude, it remains to choose $X=X^r$, $Y=\ind_{V^s>Z^d}$ and
  $Z=K^r$. We get $S^y_{N^x_t}=N^f_t$ and
  $S^z_{N^x_t}+\lfloor\frac{A^x_t}{c}\rfloor=N^i_t$ which gives the result.
\end{proof}

\subsection{Estimation of the parameters}
Theorem \ref{thm3} gives us the almost sure convergence of
$\left(\frac{N^r_t}{t}, \frac{N^f_t}{t}, \frac{N^i_t}{t}\right)$. Combining
these limits yields
\begin{align*}
  \lim_{t\rightarrow \infty} \frac{N^i_t}{N^r_t}=m_k, \;\lim_{t\rightarrow \infty} \frac{N^f_t}{N^r_t} = P_d.
\end{align*} $m_k$ depends only on $\mu$. Then, if $m_k=f(\mu)$, where
$f$ is a continuous and strictly monotone function, we get that
$\mu=f^{-1}(m_k)$. $P_d$
depends on $\mu$ and $\lambda$. Then, if $P_d=g(\mu,\lambda)$,
where $\lambda \longmapsto g(\mu,\lambda)$ is a continuous and strictly
monotone function for all $\mu$, we get that
$\lambda=g^{-1}_{\mu}(\mu,P_d)$ ($g^{-1}_{\mu}$ denotes
the inverse of $\lambda \longmapsto g(\mu,\lambda)$).

\begin{lemme} Assume that $g$ is a $C^1$ function from $U\times V$ to
  $[0,1]$, strictly monotone in $\lambda$ for all $\mu$. Then, $g^{-1}_{\mu}$,
  the inverse of $\lambda \longmapsto g(\mu,\lambda)$, is $C^1$ from $U\times
  [0,1]$ to $V$.
\end{lemme}

\begin{proof}
For all $\mu$ and $\lambda$, we have $g^{-1}_{\mu}(\mu,g(\mu,
\lambda))=\lambda$. By differentiating this equality w.r.t. $\lambda$, we get $\linebreak[4]\partial_y
g^{-1}_{\mu}(\mu,g(\mu,\lambda))\partial_y g(\mu,\lambda)=1$. Since for all $\mu$ $g$ is $C^1$
from $U\times V$ to $[0,1]$ and strictly monotone in $\lambda$, we have $\partial_y
g^{-1}_{\mu}(\mu,p)=\frac{1}{\partial_y g(\mu,g^{-1}_{\mu}(\mu,p))}$ for all
$(\mu,p) \in U\times[0,1] $. By differentiating $g^{-1}_{\mu}(\mu,g(\mu,
\lambda))=\lambda$ w.r.t. $\mu$, we get $\partial_x
g^{-1}_{\mu}(\mu,g(\mu,\lambda))+\partial_y
g^{-1}_{\mu}(\mu,g(\mu,\lambda))\partial_x g(\mu,\lambda)=0$, i.e. $\partial_x
g^{-1}_{\mu}(\mu,p)=-\partial_y
g^{-1}_{\mu}(\mu,p)\partial_x g(\mu, g^{-1}_{\mu}(\mu,p))$ for all
$(\mu,p) \in U\times[0,1] $.

\end{proof}
Let us introduce
\begin{align}\label{eq5}
\mu_t:=f^{-1}\left(\frac{N^i_t}{N^r_t}\right),\;\;
\lambda_t:=g^{-1}_{\mu}\left(\mu_t,\frac{N^f_t}{N^r_t}\right)
\end{align}
The following Theorem gives a CLT for $\sqrt{t}(\mu_t-\mu, \lambda_t-\lambda)$.

\begin{thm}\label{thm_TCL_mu_lambda}
  Assume that $m_k=f(\mu)$, $m_x=h(\mu,\lambda)$ and $P_d=g(\mu,\lambda)$, where
  $f$ is a $C^1$ strictly monotone function from $U$ to $\rset_+$, and $g$ is a $C^1$ function from $U\times V$ to
  $[0,1]$, strictly monotone in $\lambda$ for all $\mu$. Let
  $(\mu_t,\lambda_t)$ be defined by \eqref{eq5}. It holds
  \begin{align*}
    \lim_{t \rightarrow \infty} \sqrt{t}(\mu_t-\mu, \lambda_t-\lambda)
    \stackrel{law}{=} \mathcal{N}(0, \Sigma^2)
  \end{align*}
  where $\Sigma^2= ARA^T$, $R$ is given in Theorem \ref{thm_TCL} and

  \begin{align*}
    A = \frac{h(\mu,\lambda)}{f'(\mu)} \begin{pmatrix} -f(\mu)   & 1 & 0  \\
    f(\mu)\frac{\partial_{\mu}
   g(\mu,\lambda)}{\partial_{\lambda} g(\mu,\lambda)} -g(\mu, \lambda) \frac{f'(\mu)}{\partial_{\lambda}
   g(\mu,\lambda)} & -\frac{\partial_{\mu} g(\mu,\lambda)}{\partial_{\lambda} g(\mu,\lambda)} &
   \frac{f'(\mu)}{h(\mu,\lambda) \partial_{\lambda}
   g(\mu,\lambda)} \\
	\end{pmatrix}.
  \end{align*}
\end{thm}

\begin{rem}\label{rem2} If the survival function of $Y^s$ is strictly monotone in $\mu$,
  $f(\mu)$ (defined by \eqref{eq:Krfond}) is strictly monotone. By using \eqref{eq:Pdfond}, we get that $g(\mu,\lambda)=\int_0^{\infty}
  \e(F_d(D_{\overline{B}(t)}-t))n_{\mu}(dt)$. Then, If the
  survival function of $Y^d$ is strictly monotone in $\lambda$, $\lambda \longmapsto
  g(\mu,\lambda)$ is strictly monotone.
\end{rem}

\begin{proof}
  Let us first consider $\mu_t-\mu$. A Taylor expansion gives
  \begin{align}\label{eq6}
    \mu_t-\mu=f^{-1}\left(\frac{N^i_t}{N^r_t}\right)-f^{-1}(m_k)=\left(\frac{N^i_t}{N^r_t}
      -m_k\right) (f^{-1})'(\zeta_t),
  \end{align}
  where $\zeta_t$ belongs to $\left[\min\left(\frac{N^i_t}{N^r_t},m_k
    \right),\max\left(\frac{N^i_t}{N^r_t},m_k \right)\right]$. Moreover,
  $\lim_{t \rightarrow \infty} \zeta_t= m_k$ a.s..\\
  Let us now consider $\lambda_t-\lambda$.
  \begin{align}\label{eq7}
    \lambda_t-\lambda&=g^{-1}_{\mu}\left(\mu_t,\frac{N^f_t}{N^r_t}\right)-g^{-1}_{\mu}\left(\mu,P_d\right)\\
    &=g^{-1}_{\mu}\left(\mu_t,\frac{N^f_t}{N^r_t}\right)-g^{-1}_{\mu}\left(\mu,\frac{N^f_t}{N^r_t}\right)+g^{-1}_{\mu}\left(\mu,\frac{N^f_t}{N^r_t}\right)-g^{-1}_{\mu}\left(\mu,P_d
    \right),\\
    &=(\mu_t-\mu)\partial_{x}g^{-1}_{\mu}\left(\xi_t,\frac{N^f_t}{N^r_t}\right)+\left(\frac{N^f_t}{N^r_t}-P_d
      \right)\partial_{y}g^{-1}_{\mu}(\mu,\eta_t),
  \end{align}
   where $\xi_t$ belongs to $\left[\min\left(\mu_t,\mu
     \right),\max\left(\mu_t,\mu \right)\right]$ and $\eta_t$ belongs to $\left[\min\left(\frac{N^f_t}{N^r_t},P_d
     \right),\max\left(\frac{N^f_t}{N^r_t},P_d
     \right)\right]$.  Moreover,
  $\lim_{t \rightarrow \infty} \xi_t= \mu=f^{-1}(m_k)$ and $\lim_{t \rightarrow \infty}
  \eta_t= P_d$ a.s.. Combining \eqref{eq6} and \eqref{eq7} gives
   \begin{align*}
    \lambda_t-\lambda=\left(\frac{N^i_t}{N^r_t}
      -m_k\right)
    (f^{-1})'(\zeta_t)\partial_{x}g^{-1}_{\mu}\left(\xi_t,\frac{N^f_t}{N^r_t}\right)+\left(\frac{N^f_t}{N^r_t}-P_d\right)\partial_{y}g^{-1}_{\mu}(\mu,\eta_t).
  \end{align*}
  Let us introduce
  $\Gamma_t:=\frac{N^r_t}{t}-\frac{1}{m_x}$,
  $\Pi_t:=\frac{N^i_t}{t}-\frac{m_k}{m_x}$ and $\Delta_t:=\frac{N^f_t}{t}-\frac{P_d}{m_x}$. From Theorem
  \ref{thm_TCL}, we get $\sqrt{t}(\Gamma_t, \Pi_t,\Delta_t)$ converges to
  $\m N(0, R)$.
   We rewrite
  $\frac{N^i_t}{N^r_t}-m_k$ as a function of $\Gamma_t$ and $\Pi_t$.
  
  \begin{align*}
    \frac{N^i_t}{N^r_t}-m_k&=\frac{N^i_t}{t}\frac{t}{N^r_t}-m_k=\left(\frac{N^i_t}{t}-\frac{m_k}{m_x}\right)\frac{t}{N^r_t}+\frac{m_k}{m_x}\frac{t}{N^r_t}-m_k\\
    &=\Pi_t \frac{t}{N^r_t}+\frac{m_k}{m_x}\left(\frac{t}{N^r_t}-m_x\right)=\Pi_t \frac{t}{N^r_t}+m_k\frac{t}{N^r_t}\left(\frac{1}{m_x}-\frac{N^r_t}{t}\right)\\
  \end{align*}
  Then
  \begin{align*}
    \sqrt{t}\left(\frac{N^i_t}{N^r_t}-m_k\right)=h\left(\sqrt{t} \Gamma_t,\sqrt{t} \Pi_t,\frac{N^r_t}{t},m_k\right),
  \end{align*}
  where $h:(x,y,z,d) \longmapsto \frac{y}{z}-d\frac{x}{z}$. By the same type
  of computations, we get that
  \begin{align*}
    \sqrt{t}\left(\frac{N^f_t}{N^r_t}-P_d\right)=h\left(\sqrt{t} \Gamma_t,\sqrt{t} \Delta_t,\frac{N^r_t}{t},P_d\right),
  \end{align*}
  
  Then,
  \begin{align*}
    \sqrt{t}(\mu_t-\mu)&=h\left(\sqrt{t} \Gamma_t,\sqrt{t}
    \Pi_t,\frac{N^r_t}{t},m_k\right)(f^{-1})'(\zeta_t),\\
    \sqrt{t}(\lambda_t-\lambda)&=h\left(\sqrt{t} \Gamma_t,\sqrt{t}
    \Pi_t,\frac{N^r_t}{t},m_k\right)(f^{-1})'(\zeta_t)\partial_{x}g^{-1}_{\mu}\left(\xi_t,\frac{N^f_t}{N^r_t}\right)+h\left(\sqrt{t} \Gamma_t,\sqrt{t}
    \Delta_t,\frac{N^r_t}{t},P_d\right)\partial_{y}g^{-1}_{\mu}(\mu,\eta_t).
  \end{align*}
  Since $\sqrt{t}(\Gamma_t, \Pi_t,\Delta_t) \stackrel{law}{\rightarrow}\m N(0,
   R)$ and
  $\left(\frac{N^r_t}{t},\zeta_t,\xi_t,\eta_t\right)\stackrel{\mathbb{P}}{\rightarrow}
  \left(\frac{1}{m_x},m_k,f^{-1}(m_k),P_d\right)$,
  Slutsky's Theorem gives
  \begin{align*}
    \left(\begin{array}{c}\sqrt{t}(\mu_t-\mu)\\
        \sqrt{t}(\lambda_t-\lambda)\end{array}\right)\stackrel{law}{\rightarrow}&\left(\begin{array}{c}\frac{1}{f'(\mu)}
        h(G_1,G_2,\frac{1}{m_x},m_k)\\
    \frac{1}{f'(\mu)} \partial_x
    g_{\mu}^{-1}(\mu,P_d)h(G_1,G_2,\frac{1}{m_x},m_k)+ \partial_y
    g_{\mu}^{-1}(\mu,P_d)h(G_1,G_3,\frac{1}{m_x},P_d)
  \end{array}\right)\\
&=\left(\begin{array}{c} \frac{1}{f'(\mu)}(-m_x m_kG_1+m_xG_2)\\
    \frac{1}{f'(\mu)} \partial_x
    g_{\mu}^{-1}(\mu,P_d)(-m_x m_k G_1+m_xG_2)+\partial_y
    g_{\mu}^{-1}(\mu,P_d)(-m_x P_d G_1+m_xG_3)
  \end{array}
\right)\\
&=A G
\end{align*}
where $G=(G_1,G_2,G_3)^T \sim \m N(0, R)$ and
 \begin{align*}
    A = \frac{m_x}{f'(\mu)} \begin{pmatrix} -m_k   & 1 & 0  \\
    -(m_k\partial_x
   g_{\mu}^{-1}(\mu,P_d) +P_d f'(\mu)\partial_y
   g_{\mu}^{-1}(\mu,P_d)) & \partial_x g_{\mu}^{-1}(\mu,P_d) &
   \frac{f'(\mu)}{m_x}\partial_y
   g_{\mu}^{-1}(\mu,P_d)   \\
	\end{pmatrix}.
  \end{align*}
\end{proof}

\section{Applications}\label{sect:applications}

\begin{hypo}\label{hypo2} In this Section, we assume that $Y^s$ follows the gamma
  law $\Gamma(n,\mu)$, where $n \in \nset^*$ and $\mu \in \rset_+^*$, and $Y^d$ follows the exponential law of parameter
  $\lambda \in \rset_+^*$.
\end{hypo}

\begin{prop}\label{prop3}
  Under Hypothesis \ref{hypo2}, we have
  \begin{align}
	\e\left[K^r\right] &= \sum_{i=0}^{n-1} \frac{\mu^i}{i!} (-1)^i
    \left(\frac{1}{1-L}\right)^{(i)}(\mu),\label{eq:Krgam}
  \end{align}
  \begin{align*}
    1-P_d &= \left\{\begin{array}{cc}
        \frac{\mu^n}{(\mu-\lambda)^n} \left( L(\lambda) - \left(1-L(\lambda) \right)
          \sum_{i=0}^{n-1} \frac{(\mu-\lambda)^i}{i!} (-1)^i
          \left(\frac{L}{1-L}\right)^{(i)}(\mu)\right) & \mbox{ if } \lambda
        \neq \mu \notag\\
        (1-L(\mu))\frac{\mu^n}{n!}(-1)^n\left(\frac{L}{1-L}\right)^{(n)}(\mu)&
        \notag \mbox{
          if } \lambda= \mu \notag\\
      \end{array}\right.\notag\\
\end{align*}
\end{prop}

We postpone the proof of Proposition \ref{prop3} to the Appendix \ref{sect:proof}.

In order to apply Theorem \ref{thm_TCL_mu_lambda}, let us check that the
assumptions are satisfied. Since the survival function of the law
$\Gamma(n,\mu)$ is strictly monotone in $\mu$ for all $n \in \nset^*$, Remark
\ref{rem2} gives that $\mu \longmapsto f(\mu)$ and $\lambda \longmapsto g(\mu,\lambda)$ are
strictly monotone. From Proposition \ref{prop3}, we get that $f(\mu)$ is
$C^1$. Let us now check that $g$, given by $g(\mu,\lambda)=\int_0^{\infty}
\e(e^{-\lambda(D_{A(t)}-t)})\frac{\mu^n}{(n-1)!}e^{-\mu t} dt$, is a $C^1$ function on $U \times V$, i.e. we
prove that $\partial_{\mu} g(\mu,\lambda)$ and  $\partial_{\lambda}
g(\mu,\lambda)$ exist and are continuous.
\begin{itemize}
\item $\forall
\mu \in K$, a compact set included in $U$,
$\e(e^{-\lambda(D_{A(t)}-t)})\partial_{\mu} \left(\frac{\mu^n}{(n-1)!}e^{-\mu
    t}\right)$ is bounded by a positive and integrable fonction $\phi_K(t)$. Then  $\partial_{\mu}
g(\mu,\lambda)$ exist and is given by $\partial_{\mu}
g(\mu,\lambda)=\int_0^{\infty}
\e(e^{-\lambda(D_{A(t)}-t)})\frac{\mu^{n-1}}{(n-1)!}e^{-\mu t}(n-\mu t) dt$,
which is continuous.
\item $\forall
\lambda \in V$, we have $\e((D_{A(t)}-t)e^{-\lambda(D_{A(t)}-t)})\le
\e(D_{A(t)}-t)$, and $\int_0^{\infty}
\e((D_{A(t)}-t) ) \frac{\mu^n}{(n-1)!}e^{-\mu t} dt=\e(V^s-Y^s)<\infty$. Then $\partial_{\lambda}
g(\mu,\lambda)$ exist and is given by $\linebreak[4]\partial_{\lambda}
g(\mu,\lambda)=\int_0^{\infty}
\e((D_{A(t)}-t)e^{-\lambda(D_{A(t)}-t)})\frac{\mu^n}{(n-1)!}e^{-\mu t} dt$,
which is continuous.
\end{itemize}


In view of the application of Theorem \ref{thm_TCL_mu_lambda}, we need to
compute the matrix $R$ defined in Theorem \ref{thm_TCL}. The following
Proposition gives explicit formulas of its terms.

\begin{prop}\label{prop4}
  If $Y^d$ follows an exponential law, we get
  \begin{align*}
    \e\left((K^r)^2\right) &= \e\left[K^r\right] +2\sum_{i=0}^{n-1} \frac{\mu^i}{i!} (-1)^i
    \left(\frac{L}{(1-L)^2}\right)^{(i)}(\mu),\\
    \e((X^r)^2)&=\e((Y^s)^2)+\frac{2}{\lambda}\e(Z^d \ind_{V^s
      \ge Z^d})\\
    \cov(X^r, \ind_{V^s \ge Z^d})&=\e(Z^d \ind_{V^s \ge Z^d})-\e(K^r)P_d\\
     \cov(K^r, \ind_{V^s \ge
      Z^d})&=(1-P_d)\e(K^r)-\e(K^r R_d(D_{\overline{B}(Y^s)}-Y^s))\\
    \cov(K^r,X^r)&=\frac{n}{\mu}\e(K^r)_{n+1}+(\frac{1}{\lambda}-\e(X^r))\e(K^r)-\frac{1}{\lambda} \e(K^r R_d(D_{\overline{B}(Y^s)}-Y^s)),
  \end{align*}
  where $\e(Z^d \ind_{V^s \ge Z^d})
    =\frac{P_d}{\lambda}+\e(Y^s)-\frac{L'(\lambda)}{1-L(\lambda)}(1-P_d+L(\lambda)n_{\mu}(\{0\}))$
    and $\e(K^r)_{n+1}$ means that $n$ is replaced by $n+1$ in \eqref{eq:Krgam}.
  Under Hypothesis \ref{hypo2}, we have
  \begin{align*}
    &\e(K^r R_d(D_{\overline{B}(Y^s)}-Y^s))=\\
    &\left\{\begin{array}{cc}-(1-P_d)\frac{L(\lambda)}{1-L(\lambda)}+\frac{\mu^n}{(\mu-\lambda)^n} \left( \frac{L(\lambda)}{1-L(\lambda)} - \left(1-L(\lambda) \right)
          \sum_{i=0}^{n-1} \frac{(\mu-\lambda)^i}{i!} (-1)^i
          \left(\frac{L}{(1-L)^2}\right)^{(i)}(\mu)\right) & \mbox{ if }
        \lambda \neq \mu,\\
        (1-P_d)\left(\e(K^r)+\frac{L(\lambda)}{1-L(\lambda)}\right)-\frac{\mu^n}{n!}(-1)^n
        (1-L(\mu))\left(\frac{L}{(1-L)^2}\right)^{(n)}(\mu) & \mbox{ if } \lambda=\mu\\
        \end{array}\right.
    \end{align*}
\end{prop}

The proof of Proposition \ref{prop4} requires long but not difficult
computations, we leave it to the reader.

\subsection{Case $n=1$}

Proposition \ref{prop3} and \eqref{eq:Xrexp} give
  
  \begin{align*}
    \e\left[K^r\right]=f(\mu) & = \frac{1}{1-L(\mu)},\\
    P_d =g(\mu,\lambda) &= \left\{\begin{array}{cc}
        1 - \frac{\mu}{\mu-\lambda} \, \frac{L(\lambda)-L(\mu)}{1-L(\mu)} & \mbox{ if } \lambda
        \neq \mu\\
        1+\frac{\mu L'(\mu)}{1-L(\mu)} & \mbox{ if } \lambda
        = \mu\\
      \end{array}\right.\\
    \e\left[X^r\right]=h(\mu,\lambda) & =\frac{1}{\mu}+\frac{1}{\lambda}Pd.
  \end{align*}
  
  Then, $f'(\mu)=\frac{L'(\mu)}{(1-L(\mu))^2}$ and
  \begin{align*}
  \partial_{\mu} g:(\mu,\lambda) \longmapsto \left\{\begin{array}{cc}
         -\left(\frac{\lambda}{(\mu-\lambda)^2}\frac{L(\mu)-L(\lambda)}{1-L(\mu)}+\frac{\mu}{\mu-\lambda}\frac{L'(\mu)(L(\lambda)-1)}{(1-L(\mu))^2} \right) & \mbox { if } \mu \neq
  \lambda\\
 \frac{\frac{\mu}{2}L''(\mu)+L'(\mu)}{1-L(\mu)}+\frac{\mu (L'(\mu))^2}{(1-L(\mu))^2} & \mbox { if } \mu=\lambda\\
\end{array}\right.
\end{align*}

\begin{align*}
  \partial_{\lambda} g:(\mu,\lambda) \longmapsto \left\{\begin{array}{cc}
        -\left(\frac{\mu}{(\mu-\lambda)^2}\frac{L(\lambda)-L(\mu)}{1-L(\mu)}+\frac{\mu}{\mu-\lambda}\frac{L'(\lambda)}{1-L(\mu)} \right) & \mbox { if } \mu \neq
  \lambda\\
 \frac{\mu L''(\mu)}{2(1-L(\mu))} & \mbox { if } \mu=\lambda\\
\end{array}\right.
\end{align*}

\subsection{Second case : $n=2$}

\begin{hypo} In this Section, we assume that $Y^s$ follows the Gamma
  law $\Gamma(2,\frac{1}{ \mu})$, where $\mu \in \rset_+^*$, and $Y^d$ follows the exponential law of parameter
  $\lambda \in \rset_+^*$.
\end{hypo}

Proposition \ref{prop3} gives 
  
  \begin{align*}
    \e\left[K^r\right]=f(\mu) & = \frac{1-L(\mu)-\mu L'(\mu)}{(1-L(\mu))^2},\\
    P_d =g(\mu,\lambda) &= \left\{\begin{array}{cc}
        1 - \frac{\mu^2}{(\mu-\lambda)^2} \left( \frac{L(\lambda)-L(\mu)}{1-L(\mu)} + (\mu-\lambda) L'(\mu) \frac{1-L(\lambda)}{(1-L(\mu))^2}\right)& \mbox{ if } \lambda
        \neq \mu\\
        1-\frac{\mu^2 L'(\mu)}{(1-L(\mu))^2} & \mbox{ if } \lambda
        = \mu\\
      \end{array}\right.\\
    \e\left[X^r\right]=h(\mu,\lambda) & =\frac{2}{\mu}+\frac{1}{\lambda}P_d\\
  \end{align*}

Then $f'(\mu)=-\mu \frac{L''(\mu)(1-L(\mu))+2(L'(\mu))^2}{(1-L(\mu))^3}$ and

\begin{align*}
  &\partial_{\mu} g:(\mu,\lambda) \longmapsto\\
  &\left\{\begin{array}{cc}
      -\frac{\mu}{(1-L(\mu))^3(\mu-\lambda)^3}\left(-2\lambda(1-L(\mu))\left[(L(\lambda)-L(\mu))(1-L(\mu))+(\mu-\lambda)L'(\mu)(1-L(\lambda)) \right]\right. & \\
\left. +\mu (\mu-\lambda)^2(1-L(\lambda))(L''(\mu)(1-L(\mu))+2(L'(\mu))^2)  \right) & \mbox { if } \mu \neq
  \lambda\\
 -\frac{\mu}{(1-L(\mu))^3}\left(2(1-L(\mu))(\mu L'(\mu)
   L''(\mu)+(L'(\mu))^2)+(1-L(\mu))^2(L''(\mu)\right. & \\
 \left. +\frac{\mu}{3}L^{(3)}(\mu))+2\mu
 (L'(\mu))^3\right) & \mbox { if } \mu=\lambda\\
\end{array}\right.
\end{align*}

\begin{align*}
  &\partial_{\lambda} g:(\mu,\lambda) \longmapsto\\
  &\left\{\begin{array}{cc}
         -\frac{2\mu^2
         }{(\mu-\lambda)^3(1-L(\mu))^2}\left((1-L(\mu))(L(\lambda)-L(\mu))+\frac{\mu-\lambda}{2}(L'(\lambda)(1-L(\mu))\right.&
           \\
           \left.+L'(\mu)(1-L(\lambda)))-\frac{(\mu-\lambda)^2}{2}L'(\lambda)L'(\mu) \right) & \mbox { if } \mu \neq
  \lambda\\
 -\frac{\mu^2}{2(1-L(\mu))^2}(L''(\mu)L'(\mu)+\frac{1}{3}(1-L(\mu))L^{(3)}(\mu)) & \mbox { if } \mu=\lambda\\
\end{array}\right.
\end{align*}

\section{Numerical examples}\label{sect:num_ex}
In this Section we compare the asymptotical method (AM) and the ML one. We
generate datas $(Y^s_i, Y^d_i)_{i\ge 1}$ with a set of parameters $(\mu,
\lambda)$  until $\sum_{i\ge 1} X^r_i$ becomes bigger than a fixed time $T$.
We consider two cases :  deterministic inspections (see
Section \ref{sect:deterministe}) and uniform random inspections (see Section \ref{sect:alea}). In all cases, we consider that $Y^d$
follows an exponential law, and $Y^s$ follows either an exponential law or a
gamma law. From the sample $(Y^s_i, Y^d_i)_{i\ge 1}$, we get
$(N^r_T,N^i_T,N^f_T)$, and \eqref{eq5} gives $(\mu_T,\lambda_T)$. Moreover,
Theorem \ref{thm_TCL_mu_lambda} gives a confidence interval for this
approximation of $(\mu,\lambda)$. We compare these results with the ones given
by the maximum likelihood method.

\subsection{Deterministic inspections}
We assume that the random variable $C$ is constant, equal to $c$. In this case
we have $L(s)=e^{-sc}$.

\subsubsection{Exponential law for $Y^s$}\label{sect:deterministe}

We assume that $Y^s$ follows the law
$\mathcal{E}(\mu)$. 
We have generated datas with the following parameters
\begin{align}\label{eq:param}
  \mu=10^{-3},\;\; \lambda=5.10^{-4},\;\; c=1000.
\end{align} We have obtained $N^r_T= 33501$, $N^f_T=8255$, and $N^i_T=53116$
at time $T=50001908$. Table \ref{tab1} compares both estimators (ML and AM)
and their $95 \%$
confidence intervals for $\mu$ and $\lambda$. Both methods
are very precise and give almost the same values for the estimators and
confidence intervals. However, the asymptotical method is much faster than the
likelihood one : the computational time of the ML method is $62.23$s, whereas
the computational time of the asymptotical method is less than $10^{-4}$.

\begin{rem}\label{rem7}
   The relative error on $\mu$ is about $0.3\%$ and the
one on $\lambda$ is about $0.8\%$. Indeed, we have at our disposal more datas to
calibrate $\mu$ than to calibrate $\lambda$.
\end{rem}
\begin{table}[htbp]
  \begin{center}
\begin{tabular}{|c|c|c|}
  \hline
  Method & ${\mu}$ and CI & ${\lambda}$ and CI \\
  \hline
  ML &  0.000996215 $[0.0009850  ,0.0010074 ]$ & 0.000504164 $[0.0004928  , 0.0005155]$  \\
  \hline
  AM & 0.000996184 $[0.0009532  ,0.0010392  ]$ & 0.000504197  $[0.0004870 , 0.0005214   ]$ \\
  \hline
\end{tabular}
\caption{Comparison of AM and ML when C is constant and $Y^s \sim
  \mathcal{E}$}\label{tab1}
\end{center}
\end{table}

Figures \ref{fig4} and \ref{fig5} plot the evolution of the parameters with
respect to time $t$, when $t$ varies in $[0,T]$
($(\hat{\mu}_t,\hat{\lambda}_t)_t$ represents the evolution of the ML estimators).   Both estimators evolve in the
same way. Moreover, the convergence is quite fast. As said in Remark
\ref{rem7}, we observe that $(\mu_t)_t$ is more precise than $(\lambda_t)_t$.

\begin{figure}[ht]
  \centering
  \pgfimage[width=15cm]{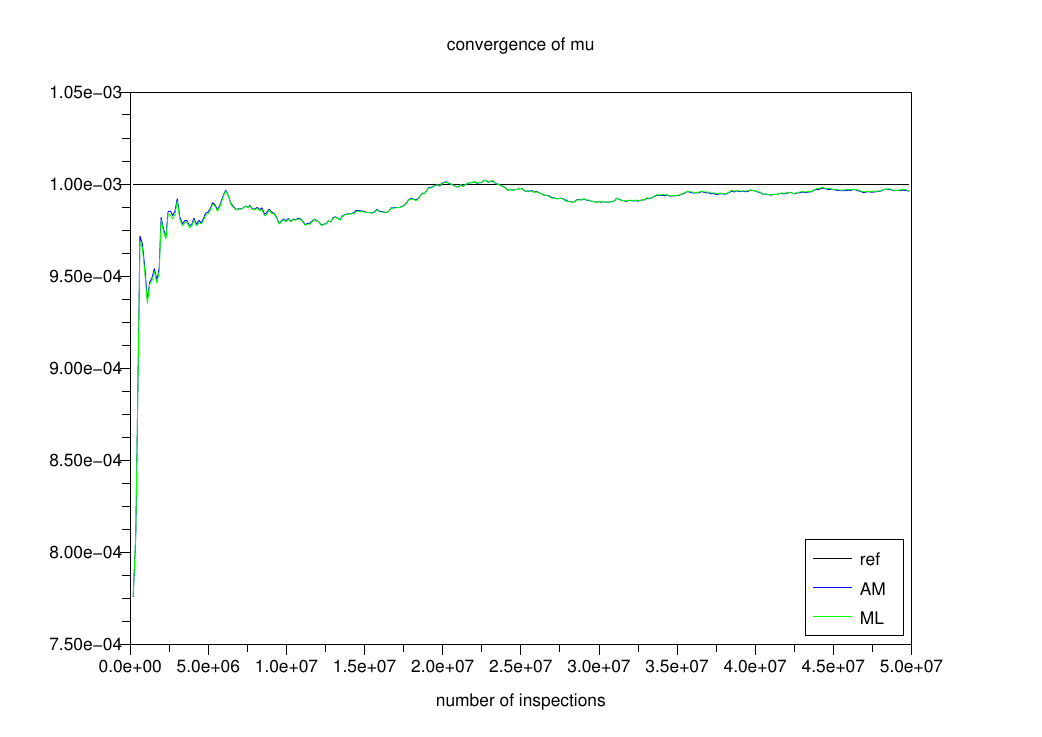}
  \caption{Convergence of $\mu_t$ and $\hat{\mu}_t$}
  \label{fig4}
\end{figure}

\begin{figure}[ht]
  \centering
  \pgfimage[width=15cm]{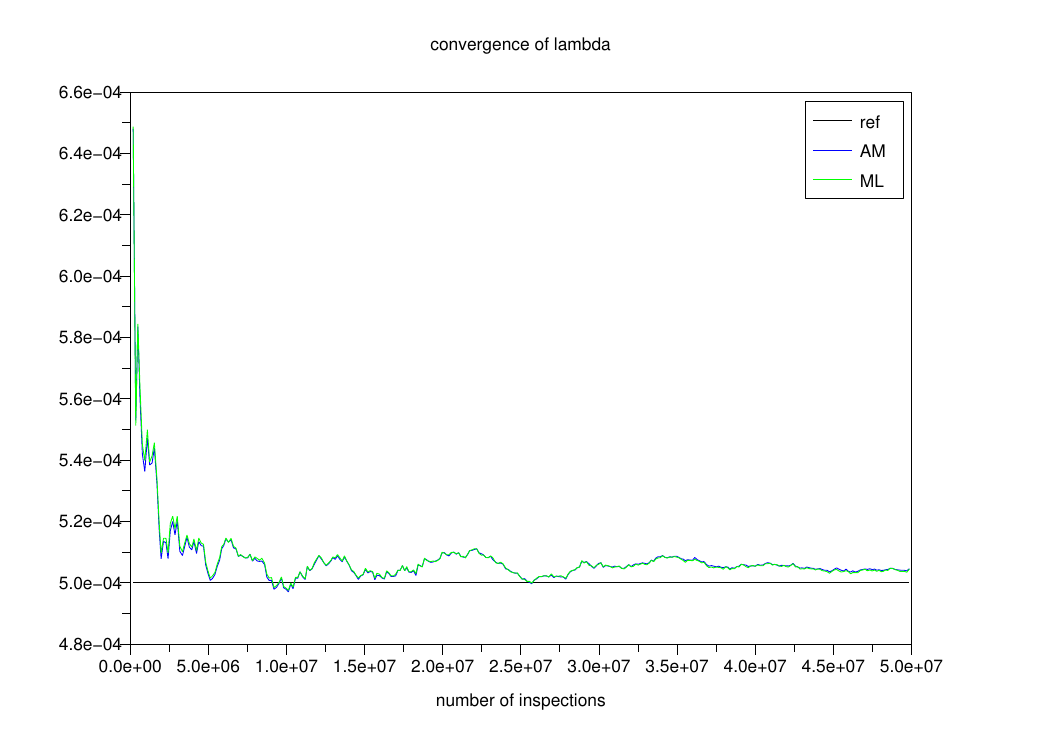}
  \caption{Convergence of $\lambda_t$ and $\hat{\lambda}_t$}
  \label{fig5}
\end{figure}

\subsubsection{Gamma law for $Y^s$}

We assume that $Y^s$ follows the law
$\Gamma(2,\mu)$ and parameters are given by \eqref{eq:param}.
We obtain $N^r_T=20668$, $N^f_T=4369$, and $N^i_T=51503$ at time
$T=50002058$. Table \ref{tab2} compares both estimators (ML and AM) and their
$95 \%$ confidence intervals for $\mu$ and $\lambda$. As in the case of an 
exponential law for $Y^s$, both methods
are very precise and give almost the same values for the estimators and
confidence intervals.

\begin{table}[htbp]
\begin{center}
\begin{tabular}{|c|c|c|}
  \hline
  Method & ${\mu}$ and CI & ${\lambda}$ and CI  \\
  \hline
  ML &  0.0010052 $[0.0009953,0.0010151]$ & 0.0004992 $[0.0004784  ,0.0005144  ]$ \\
  \hline
  AM &  0.0010054 $[0.0009742,0.0010366]$ & 0.0004918 $[0.0004789,0.0005047 ]$   \\
  \hline
\end{tabular}
\end{center}
\caption{Comparison of AM and ML when C is constant and $Y^s \sim
  \Gamma(2,\mu)$}\label{tab2}
\end{table}


\subsection{Random inspections}\label{sect:alea}
We assume that $C$ follows a uniform law on $[c-h,c+h]$. In this case, we have $L(s)=e^{-sc}\frac{\sinh{(sh)}}{sh}$.

\subsubsection{Exponential law for $Y^s$}

We assume that $Y^s$ follows the law
$\mathcal{E}(\mu)$. We have generated datas with the following parameters
\begin{align}\label{eq:param2}
  \mu=10^{-3},\;\; \lambda=5.10^{-4},\;\; c=1000,\;\; h=100.
\end{align}
We obtain $N^r_T= 33613$, $N^f_T=8278$, and $N^i_T=53133$ at time
$T=50001271$. Table \ref{tab3} compares both estimators (ML and AM) and their
$95 \%$ 
confidence intervals for $\mu$ and $\lambda$. As in case of
deterministic inspections, both methods
are very precise and give almost the same values for the estimators and
confidence intervals.
\begin{table}[htbp]
\begin{center}
\begin{tabular}{|c|c|c|c|}
  \hline
  Method & ${\mu}$ and CI & ${\lambda}$ and CI  \\
  \hline
  ML & 0.0010030 $[0.0009918,0.0010142]$ & 0.000502 $[0.0004907,0.0005133]$ \\
  \hline
  AM & 0.0010030 $[0.0009602  ,0.0010458  ]$  & 0.0005021 $[0.0004849 , 0.0005193  ]$ \\
  \hline
\end{tabular}
\end{center}
\caption{Comparison of AM and ML when C is random and $Y^s \sim
  \mathcal{E}$}\label{tab3}
\end{table}


\subsubsection{Gamma law for $Y^s$}
We assume that $Y^s$ follows the law
$\Gamma(2,\mu)$, with parameters given by \eqref{eq:param2}. We obtain
$N^r_T= 20470$, $N^f_T=4452$, and $N^i_T=51522$ at time $T=50000355$. Table
\ref{tab4} compares both estimators (ML and AM) and their $95\%$
confidence intervals for $\mu$ and $\lambda$. As in case of
deterministic inspections, both methods
are very precise and give almost the same values for the estimators and
confidence intervals.

\begin{table}[htbp]
\begin{center}
\begin{tabular}{|c|c|c|c|}
  \hline
  Method & ${\mu}$ and CI & ${\lambda}$ and CI \\
  \hline
  ML & 0.000996 $[0.0009862,0.0010058]$ & 0.000505  $[0.0004895,0.0005205]$ \\
  \hline
  AM & 0.0009964  $[0.0009661  ,0.0010267  ]$ & 0.0005064 $[0.0004934 ,
  0.0005194 ]$  \\
  \hline
\end{tabular}
\end{center}
\caption{Comparison of AM and ML when C is random and $Y^s \sim
  \Gamma(2, \mu)$}\label{tab4}
\end{table}

\section{Conclusion}
We have presented a new method to estimate the parameters of the laws of the
transition times of a $3$-state deteriorating system. Each law depends on one
parameter. However, Theorem \ref{thm3} gives the almost sure convergence of
three quantities of interest, which allows to estimate three unknown parameters.
Then, our method can be extended to a more general setting. This topic, as
well as possible extensions to $k$-state deteriorating systems ($k \ge 4$)
will be treated in a future work.
\appendix

\section{Proof of Proposition \ref{prop3}}\label{sect:proof}
Firstly we give the following lemma, useful to show Propositions \ref{prop3}
and \ref{prop4}. Since its proof requires long but not
difficult computations, we leave it to the reader.
\begin{lemme}\label{lem1}
  Under Hypothesis \ref{hypo2}, we have
  \begin{align*}
    \sum_{k\ge 1} \e\left[ e^{-\lambda D_k} \int_0^{D_k} e^{\lambda t}
      n_{\mu}(dt)\right] & =\left\{\begin{array}{cc} \frac{\mu^n}{(\mu-
      \lambda)^n}\left(\frac{L(\lambda)}{1-L(\lambda)}-\sum_{i=0}^{n-1}\frac{(\mu-\lambda)^i}{i!}
      (-1)^i \left(\frac{L}{1-L}\right)^{(i)}(\mu) \right) & \mbox{ if }
    \lambda \neq \mu,\\
    \frac{\mu^n}{n!}(-1)^n\left(\frac{L}{1-L}\right)^{(n)}(\mu)& \mbox{ if }
    \lambda = \mu\\
    \end{array}\right.
\end{align*}
and
\begin{align*}
    \sum_{k\ge 1} k \e\left[ e^{-\lambda D_k} \int_0^{D_k} e^{\lambda t} n_{\mu}(dt)\right] & =\left\{\begin{array}{cc} \frac{\mu^n}{(\mu-
      \lambda)^n}\left(\frac{L(\lambda)}{(1-L(\lambda))^2}-\sum_{i=0}^{n-1}\frac{(\mu-\lambda)^i}{i!}
      (-1)^i \left(\frac{L}{(1-L)^2}\right)^{(i)}(\mu) \right) & \mbox{ if }
    \lambda \neq \mu,\\
    \frac{\mu^n}{n!}(-1)^n\left(\frac{L}{(1-L)^2}\right)^{(n)}(\mu)& \mbox{ if }
    \lambda = \mu\\
    \end{array}\right.
\end{align*}
\end{lemme}

Let us first prove \eqref{eq:Krgam}. From \eqref{eq:Krfond}, we have
\begin{equation*}
	\e\left[K^r\right] = \sum_{k\geq 0} \e\left[R^s(D_k)\right] = \sum_{k\geq 0} \sum_{i=0}^{n-1} \frac{\mu^i}{i!} \e\left[D_k^i e^{-\mu D_k}\right] = \sum_{k\geq 0} \sum_{i=0}^{n-1} \frac{\mu^i}{i!} (-1)^i L_k^{(i)}(\mu).
\end{equation*}
Then
\begin{equation*}
	\e\left[K^r\right] = \sum_{i=0}^{n-1} \frac{\mu^i}{i!} (-1)^i \sum_{k\geq 0} L_k^{(i)}(\mu) = \sum_{i=0}^{n-1} \frac{\mu^i}{i!} (-1)^i \left(\sum\nl_{k\geq 0} L_k\right)^{(i)}(\mu),
\end{equation*}
and the result follows.
The second result ensues from \eqref{eq:Pdexp} and Lemma \ref{lem1}. 

\bibliography{ref}
\bibliographystyle{amsalpha}
\end{document}